\documentclass[11pt,reqno]{amsart}

\usepackage[initials]{amsrefs}
\usepackage{amssymb,graphicx,amscd,color,amsmath,amsfonts,amssymb,latexsym,amsthm,units}
\usepackage[all]{xy}

\newtheorem{theorem}{Theorem}[section]
\newtheorem{proposition}[theorem]{Proposition}

\newtheorem{remark}[theorem]{Remark}

\newtheorem{lemma}[theorem]{Lemma}

\newtheorem{definition}[theorem]{Definition}

\def\N{\mathbb N}
\def\C{\mathbb C}
\def\R{\mathbb R}

\def\Q{\mathbb Q}

\def\ve{\varepsilon}

\def\al{\alpha}
\def\la{\lambda}

\def\ovl{\overline}

\begin{document}

\title[{Space-filling curves: topological and algebraic structure}]
{The set of space-filling curves: topological and algebraic structure}

\date{}

\author[Bernal-Gonz\'alez, Calder\'on-Moreno and Prado-Bassas]{L.~Bernal-Gonz\'alez, M.C.~Calder\'on-Moreno and J.A.~Prado-Bassas}
\address{Departamento de An\'alisis Matem\'atico,\newline\indent
Facultad de Matem\'aticas, \newline\indent
Universidad de Sevilla, \newline\indent
Apdo. 1160, Avenida Reina Mercedes, \newline\indent
Sevilla, 41080, Spain.}
\email{lbernal@us.es, mccm@us.es {\rm and} bassas@us.es}

\thanks{{\sl This paper is dedicated to Professor Jos\'e Bonet Solves on his 60th birthday.}}
\thanks{The authors have been partially supported by the Plan Andaluz de Investigaci\'on de la
Junta de Andaluc\'{\i}a FQM-127 Grant P08-FQM-03543 and by MEC Grant MTM2012-34847-C02-01.}


\begin{abstract}

In this paper, a study of topological and algebraic properties of two families of functions from the
unit interval $I$ into the plane $\R^2$ is performed. The first family is the collection of all Peano curves,
that is, of those continuous mappings onto the unit square. The second one is the bigger set of all
space-filling curves, i.e.~of those continuous functions $I \to \R^2$ whose images have positive Jordan content.
Emphasis is put on the size of these families, in both topological and algebraic senses, when endowed with
natural structures.
\end{abstract}

\maketitle




\section{Introduction}

\quad In 1890 G.~Peano  {\cite{peano}} showed the existence of an astonishing mathematical object, namely, a curve filling the unit square.
To be more precise, he constructed a {\it continuous surjective} mapping $I \to I^2$, where $I = [0,1]$ is the closed unit
interval in the real line $\R$ and $I^2 = [0,1] \times [0,1]$.

\vskip .15cm

Lebesgue \cites{gelbaumolmsted1964,gelbaumolmsted2003,lebesgue1904} was probably
the first to show an example of a function $f: \R \to \R$ that is {\it surjective} in a strong sense.
Specifically, it satisfies $f(J) = \R$ for every nondegenerate interval $J$.
Since then, many families of surjections $\R \to \R$,
even in much stronger senses, have been presented (see {\cite{jordan1998,jones1942,foran1991}}).
Nevertheless, each of these functions is nowhere continuous. Of course, by using a bijection $\R \to I^2$ or $\R \to \R^2$,
surjections $\R \to I^2$ or $\R \to \R^2$ (or even $I \to I^2$) can be constructed, but
their continuity is far from being guaranteed.

\vskip .15cm

Peano's result admits a topological extension, and in fact a topological characterization, which is given by
the {\it Hahn--Mazurkiewicz theorem} (see e.g.~\cite[Theorem 31.5]{willard} or \cite{HoY}): a Hausdorff topological
space \,$Y$ is a continuous image of the unit interval if and only if it is a compact, connected, locally connected,
and second-countable space. Such a space $Y$ is called a {\it Peano space.}
Equivalently, by well-known metrization theorems, a Peano space is a compact, connected, locally connected metrizable
topological space. Given two topological spaces $X$ and $Y$, the set of continuous (continuous surjective, resp.)
mappings $X \to Y$ will be denoted by \,$C(X,Y)$ ($CS(X,Y)$, resp.). Then the family of Peano curves is \,${\mathcal P}
:= CS(I,I^2)$. If \,$Y$ is a Peano space, we also denote ${\mathcal P}_Y := CS(I,Y)$, so that ${\mathcal P} = {\mathcal P}_{I^2}$.

\vskip .15cm

There are several extensions of the notion of Peano curve on $\R^N$, with $N \ge 2$.
Since the case $N=2$ is illuminating enough, we will restrict ourselves to it. For ins\-tan\-ce, in \cite{sagan},
the next notion is given. By \,$c(A)$ \,it is denoted the Jordan content of a Jordan measurable set $A \subset \R^2$
(see Section 2 for definitions).

\begin{definition} \label{def spacefilling curve}
{\rm We say that a continuous function \,$\varphi : I \to \R^2$ \,is a {\it space-filling curve}
provided that \,$\varphi (I)$ is Jordan measurable and \,$c(\varphi (I)) > 0$.}
\end{definition}

We can relax this condition by defining a {\it $\la$-space-filling curve} --where $\la$
denotes Lebesgue measure on $\R^2$-- as a continuous function $f:I \to \R^2$ with $\lambda (f(I)) > 0$.
This is not equivalent to the former definition; as a matter of fact, Osgood \cite{Osgood,sagan}
constructed in 1903 a Jordan curve, that is, a continuous {\it injective}
\,function $\psi :I \to \R^2$, such that $\la (\psi (I)) > 0$; here $\psi (I)$ cannot be Jordan measurable.
Other related notions can be found in \cite{moramira} and \cite{Ubeda2006}.
The symbol \,$\mathcal{SF}$ \,will stand for the set of all space-filling curves in the sense of Definition
\ref{def spacefilling curve}.

\vskip .15cm

The main concern of this paper is to study both families ${\mathcal P}$ and $\mathcal{SF}$ from
the topological-algebraic point of view, with special emphasis on the size of such sets,
rather than on properties of individual members of them. For this, ${\mathcal P}$ and $\mathcal{SF}$
are supposed endowed with their natural topologies. The diverse notions of largeness that will be
considered, together with other preliminaries, are compiled in Section 2. Finally, Sections 3 and 4
contain our main results, which demonstrate
the existence of large --topological or algebraic-- structures within the mentioned families.

\section{Topological and linear size concepts}

\quad When dealing with subsets of a metric space $(X,d)$, one way to describe their smallness is by
means of the notion of porosity, introduced by Dolzenko \cite{Dolzenko} in 1967 for the real line and
generalized by Zaj\'{\i}$\check{\rm c}$ek \cite{Zajicek}. Here we use a slightly stronger notion of porosity \cite{Zamfirescu}.
By $B(x,r)$ we denote the open ball in $X$ with center $x \in X$ and radius $r > 0$,
while $\ovl{A}$ stands for the closure of a set $A$ in a given topological space.

\begin{definition}
{\rm A subset $A$ in a metric space $(X,d)$ is called {\it porous} if there is $\al > 0$ such that
for each $x \in X$ and each $\ve > 0$ there exists $y \in B(x,\ve )$ such that
$$
B(y,\al \, d(x,y)) \cap A = \varnothing .
$$
If the above number $\al > 0$ can be chosen as close to $1$ as we wish then $A$ is called
{\it strongly porous.}}
\end{definition}

It is well known that any porous set $A$ is nowhere dense, that is, its interior $\ovl{A}^0 = \varnothing$.
In fact, porosity is a notion strictly stronger than nowhere density. Porosity will be considered
in the context of Peano curves.

\vskip .15cm

In a completely metrizable topological space $X$ (so that Baire's theorem applies),
one way to describe smallness or largeness is by meagerness: a subset $A \subset X$ is said to
be {\it meager} or {\it of first category} if it is a countable union of nowhere dense subsets;
and $A$ is called {\it residual} if the complement $X \setminus A$ is meager or, equivalently, if $A$
is a countable intersection of dense open sets. Hence, in a topological sense, a residual set is very large,
and in fact the existence of many ``strange'' mathematical objects has been {stated} by proving
that their set is residual (in some appropriate topological space). Incidentally, each set of such mathematical
objects turns to be huge.

\vskip .15cm

A different, recently introduced approach to study the size of a family of objects arises from the
{\it theory of lineability.} The following notions can be found in \cite{AronGS,APS_Studia,bartoszewiczglab2012b,BarGP,bernal2010,BerPS,gurariyquarta2004,seoane2006}.

\begin{definition}
{\rm If $X$ is a vector space, $\alpha$ is a cardinal number and $A \subset X$, then $A$ is said to be:
\begin{enumerate}
\item[$\bullet$] {\it lineable} if there is an infinite dimensional vector space $M$ such that $M \setminus \{0\} \subset A$,
\item[$\bullet$] {\it $\alpha$-lineable} if there exists a vector space $M$ with dim$(M) = \alpha$ and $M \setminus \{0\} \subset A$
(hence lineability means $\aleph_0$-lineability,
where $\aleph_0 = {\rm card}\,(\N )$, the cardinality of the set of positive integers), and
\item[$\bullet$] {\it maximal lineable} in $X$ if $A$ is ${\rm dim}\,(X)$-lineable.
\end{enumerate}
If, in addition, $X$ is a topological vector space, then $A$ is said to be:
\begin{enumerate}
\item[$\bullet$] {\it dense-lineable} in $X$
whenever there is a dense vector subspace $M$ of $X$ satisfying $M \setminus \{0\} \subset A$,
\item[$\bullet$] {\it maximal dense-lineable} in $X$
whenever there is a dense vector subspace $M$ of $X$ satisfying $M \setminus \{0\} \subset A$ and
dim$\,(M) =$ dim$\,(X)$, and
\item[$\bullet$] {\it spaceable} in $X$ if there is a closed infinite dimensional vector
subspace $M$ such that $M \setminus \{0\} \subset A$.
\end{enumerate}
When $X$ is a topological vector space contained in some (linear) algebra then $A$ is called:
\begin{enumerate}
\item[$\bullet$] {\it algebrable} if there is an algebra \,$M$ so
    that $M \setminus \{0\} \subset A$ and $M$ is infinitely generated, that is, the cardinality of any system of generators of \,$M$ is infinite, and
\item[$\bullet$] {\it strongly algebrable} if, in addition, the algebra $M$ can be taken free.
\end{enumerate}
}
\end{definition}

Note that if $X$ is contained in a commutative algebra then a set $B \subset X$ is a generating set of some {\it free} algebra contained in $A$
if and only if for any $N \in \N$, any nonzero polynomial $P$ in $N$ variables without constant term and any distinct $f_1,...,f_N \in B$, we have
$P(f_1,...,f_N) \ne 0$ and $P(f_1,...,f_N) \in A$. Observe that strong-algebrability $\Rightarrow$ algebrability $\Rightarrow$ lineability, and none of these implications can be reversed, see \cite{bartoszewiczglab2012b} and \cite[p.~74]{BerPS}.

\vskip .15cm

From Peano's result, it is not difficult to extend his filling curve $I \to I^2$ to a continuous surjective
function $\R \to \R^2$. This can be generalized as to obtain that $CS(\R^m,\R^n) \ne \varnothing$ for all $m,n \in \N$.
In fact, Albuquerque, Bernal, Ord\'o\~{n}ez, Pellegrino and Seoane \cite{nga,AlbBerPelSeo,BernalOrd} have recently
shown that $CS(\R^m,\R^n)$ is maximal dense-lineable and spaceable in $C(\R^m,\R^n)$, and that $CS(\R^m,\C^n)$ is
strongly $\mathfrak c$-algebrable (here $\mathfrak c$ stands for the cardinality of the continuum, $\C$ denotes the
complex field, and the algebra structure of $C(\R^m,\C^n)$ is defined coordenatewise). In \cite{AlbBerPelSeo}, the
lineability of the families $CS(\R^m,Y)$, where $Y$ represents some relevant subspaces of infinite dimensional
Euclidean spaces, is also analyzed. To summarize, these diverse CS-families are large in several algebraic (or
topological-algebraic) senses.

\vskip .15cm

It must be said that the mentioned results in \cite{nga,AlbBerPelSeo,BernalOrd}
were the inspiration for the present paper, but there is an important point which is why the methods given in them
cannot be directly reproduced in our setting. Namely, our starting space is the {\it compact} interval $I$. Hence $f(I)$ is compact for
any continuous mapping on $I$, so $f(I)$ is never ``too much large''. Furthermore, our family $\mathcal P$ is not even
stable under scaling, which causes that the study of lineability of $\mathcal P$ makes no sense.

\vskip .15cm

In order to investigate the algebraic structure of $\mathcal P$, let us introduce the following concept.

\begin{definition}
{\rm Assume that $(X,*)$ is a semigroup and that $A \subset X$. We say that $A$ is {\it semigroupable} whenever
there exists an infinitely generated semigroup $G \subset A$.}
\end{definition}

\begin{remark}
{\rm We recall that a semigroup \,$G$ \,is called infinitely generated whenever it is not finitely generated, that is,
there does {\it not} exist a finite set $F \subset X$ such that every $x \in G$ can be written as a finite product
$x = a_1^{m_1} * \cdots * a_p^{m_p}$, with $a_1, \dots , a_p \in F$ and $m_1, \dots , m_p \in \N$
(of course, $p$, $a_i$ and $m_i$ depend upon $x$). The $a_i$'s are not necessarily different: take into account
that $(X,*)$ might be noncommutative. Nevertheless, the semigroup $X$ that will be considered in this paper is
$C(I,I^2)$, where the operation $*$ is the coordenatewise multiplication, which is commutative. Hence the $a_i$'s
can be taken different in this case.}
\end{remark}

Recall that if $E$ is a Banach space then a sequence $\{x_n\}_{n \ge 1}$
is called a basic sequence whenever it is a Schauder basis of its generated closed vector subspace, that is, whenever every vector
$x \in \ovl{\rm span} \{x_n\}_{n \ge 1}$ can be uniquely represented by a series $x = \sum_{n \ge 1} \la_n x_n$ converging in the norm $\| \cdot \|$ of $E$. By Nikolskii's theorem (see for instance \cite{diestel1984}), a sequence $\{x_n\}_{n \ge 1} \subset E \setminus \{0\}$ is basic if and only if there is a constant $\alpha \in (0,+\infty )$ such that, for every pair
$r,s \in \N$ with $s \ge r$ and every finite sequence of scalars $a_1, \dots , a_s$, one has
$$\big\|\sum_{n=1}^r a_n x_n\big\| \leq \alpha \big\|\sum_{n=1}^s a_n x_n\big\| .$$
For any $N \in \N$, we will consider  the norm
$\|f\| = \sup_{t \in I} \|f(t)\|_1$ in the space $C(I,\R^N)$, which makes it a Banach space; here $\| \cdot \|_1$ represents
the $1$-norm in $\R^N$, given by $\|(x_1, \dots ,x_N)\|_ 1 = \max_{1 \le i \le N} |x_i|$.
In Section 4 the next lemma --which is a direct application of Nikolskii's theorem-- will be needed.

\begin{lemma} \label{Lemma-disjointsupports}
Assume that $\{f_n\}_{n \ge 1}$ is a sequence in $C(I,\R^N) \setminus \{0\}$ such that the supports $\{t \in I: \, f_n(t) \ne 0\}$
$(n=1,2, \dots )$ are mutually disjoint. Then $\{f_n\}_{n \ge 1}$ is a basic sequence in $C(I,\R^N)$.
\end{lemma}

The next assertion --which is proved in \cite[Theorem 2.3]{BernalOrd}
(see also \cite{arongarciaperezseoane2009,bernal2008,bernal2010})-- will be
useful in Section 4 to get dense-lineability from mere lineability.

\begin{theorem} \label{Thm-denselineability-OrdonezLBG}
Assume that $E$ is a metrizable separable topological vector space and that \,$\al$ is an infinite cardinal number.
Let $A,B \subset E$ be two subsets such that $A$ is $\al$-lineable, $B$ is dense-lineable, $A \cap B = \varnothing$
and $A+B \subset A$. Then $A$ contains a dense vector space $M$ with {\rm dim}$(M) = \al$.
\end{theorem}

The following elementary lemma will be used repeatedly along Sections 3--4.

\begin{lemma} \label{Peano extrema prescribed}
Let \,$Y$ be a Peano space and \,$[a,b]$ be a closed interval in $\R$. Given $u,v \in Y$, there is a
mapping \,$\Phi \in CS([a,b],Y)$ such that \,$\Phi (a) = u$ and \,$\Phi (b) = v$.
\end{lemma}

\begin{proof}
By the Hahn--Mazurkiewicz theorem, we can select a mapping $f \in {\mathcal P}_Y$.
Since Peano spaces are arcwise connected \cite[Theorem 31.2]{willard}, there are continuous mappings $g:[0,1/3] \to Y$
and $h:[2/3,1] \to Y$ satisfying $g(0) = u$, $g(1/3) = f(0)$, $h(2/3) = f(1)$ and $h(1) = v$.
Define $\varphi : I \to Y$ as
$$
\varphi (t) =  \left\{
\begin{array}{ll}
                 g(t)  & \mbox{if } 0 \le t < {1/3}  \\
                 f(3t-1) & \mbox{if } {1/3} \le t \le {2/3}  \\
                 h(t)  & \mbox{if }  {2/3}  \le t \le 1.
\end{array} \right.
$$
Then it is evident that the mapping $\Phi :[a,b] \to Y$ given by $\Phi (t) = \varphi \big({t-a \over b-a} \big)$
does the job.
\end{proof}

Finally, let us recall a number of concepts concerning the Jordan mea\-su\-ra\-bi\-li\-ty. Assume that $S$ is a bounded subset
of $\R^2$. Then the inner Jordan content and the outer Jordan content of $S$ are respectively given by the
following lower and upper Riemann integrals:
{$$
\underline{c}(S) = \underline{\int} \chi_S \,dxdy, \,\,\, \ovl{c}(S) = \ovl{\int} \chi_S \,dxdy,
$$}
where $\chi_S$ denotes the characteristic function of $S$. The set $S$ is said to be {\it Jordan measurable}
provided that $\underline{c}(S) = \ovl{c}(S)$, in which case their common value \,$c(S)$ \,is called the
{\it Jordan content} \,of $S$. This happens if and only if $\chi_S$ is Riemann integrable, and if and only if
$\lambda (\partial S) = 0$ ($\partial S$ denotes the boundary of $S$). Moreover, in this case, $S$ is Lebesgue
measurable and $c(S) = \la (S)$.

\section{The family of Peano curves}

\quad A natural, complete distance on the space $C(I,\R^2)$ is given by
\begin{equation}\label{Eq1}
\rho (f,g) = \sup_{t \in I} d_\infty (f(t),g(t)),
\end{equation}
that generates the topology of uniform convergence on $I$. Here $d_\infty$ is the metric on $\R^2$
resulting from the $1$-norm $\| \cdot \|_1$, that is, $d_\infty ((a,b),(c,d)) = \max \{|a-c|,|b-d|\}$ (other equivalent, even similar,
metrics are available on $\R^2$, but $d_\infty$ is more convenient for the sake of calculations).
Of course, $\mathcal P$ is a very small subset of $C(I,\R^2)$. The main reason for it is that $f(I) = I^2$
for each $f \in {\mathcal P}$. This is why it is more natural to consider \,$\mathcal P$ as a topological subspace
of $C(I,I^2)$ rather than of $C(I,\R^2)$. Observe that, due to the fact that uniform convergence entails pointwise convergence,
$C(I,I^2)$ (endowed with the distance $\rho$ induced from $C(I,\R^2)$) is closed in $C(I,\R^2)$
(in fact, $C(I,A)$ is closed in $C(I,\R^2)$, for every closed set $A \subset \R^2$), so it is a complete metric space.

\vskip .15cm

For a general Peano space \,$Y$, it will be endowed with a fixed distance \,$d$ \,generating
its topology (note that, as \,$Y$ is compact, any distance generating its topology is complete).
Then, just by changing $d_\infty$ to $d$, the expression \eqref{Eq1} above defines a complete distance on $C(I,Y)$.
Observe that, since \,$Y$ is metrizable and arcwise connected, it is uncountable as soon as it possesses more than one point;
in fact, every nonempty open subset of \,$Y$ is uncountable.
In the following theorem, we gather some topological or me\-tri\-cal properties of \,$\mathcal P$.
We use standard notation for a metric space $(X,D)$: $B_D(x_0,r)$ and $\ovl{B}_D(x_0,r)$
will stand, respectively, for the open ball and the closed ball with center $x_0 \in X$ and radius $r > 0$.

\begin{theorem}
Assume that \,$Y$ is a Peano space. We have:
\begin{enumerate}
\item[\rm (a)] ${\mathcal P}_Y$ is closed in \,$C(I,Y)$. In particular, ${\mathcal P}_Y$ is a completely metrizable space.
\item[\rm (b)] If \,$Y$ has at least two points then \,${\mathcal P}_Y$ is not compact.
\item[\rm (c)] Assume that \,$Y$ has at least two points and that there is \,$y_0 \in Y$ satisfying the following
               property: given a neighborhood \,$U$ of \,$y_0$, there exists a neighborhood \,$V$ of \,$y_0$ such that $V \subset U$
               and $V \setminus \{y_0\}$ is arcwise connected. Then \,${\mathcal P}^0_Y = \varnothing$. Hence \,${\mathcal P}_Y$ is nowhere dense in $C(I,Y)$.
\item[\rm (d)] In the case $Y = I^2$, the Peano family \,${\mathcal P}_Y = {\mathcal P}$ is strongly porous in $C(I,I^2)$.
\end{enumerate}
\end{theorem}

\begin{proof}
(a) Let $F \in C(I,Y)$ and $\{f_n\}_{n \ge 1}$ be a sequence in \,${\mathcal P}_Y$ with $f_n \to F$.
Fix $y \in Y$. Then there is a sequence $\{t_n\}_{n \ge 1} \subset I$ such that $f_n(t_n) = y$ for all $n \in \N$.
Since $I$ is compact, we can take out a subsequence $\{t_{n_k}\}_{k \ge 1}$ converging to some point $t_0 \in I$.
The continuity of $F$ yields $\alpha_k := d(F(t_{n_k}),F(t_0)) \to 0$ as $k \to \infty$.
From the triangle inequality,
$$
d(y,F(t_0)) \le  d(f_{n_k}(t_{n_k}),F(t_{n_k})) + d(F(t_{n_k}),F(t_0)) \le \rho (f_{n_k},F) + \al_k \longrightarrow 0.
$$
Hence \,$d(y,F(t_0)) = 0$ \,or, that is the same, $F(t_0) = y$. Since \,$y$ \,was arbitrary, $F$ is surjective, that is, $F \in {\mathcal P}_Y$.
Therefore \,${\mathcal P}_Y$ is closed.

\vskip .15cm

\noindent (b) Choose $y,z \in Y$ with $y \ne z$. Set $\ve := d(y,z) > 0$ and fix $\delta \in (0,1)$.
By Lemma \ref{Peano extrema prescribed}, there exists a continuous surjective function $\Phi : [0,\delta /2] \to Y$
with $\Phi (0) = y = \Phi (\delta /2)$. In particular, there is $v \in [0,\delta /2]$ such that $\Phi (v) = z$.
By extending \,$\Phi$ \,as \,$y$ \,on $(\delta /2,1]$ and setting $u := 0$, we obtain points $u,v \in I$ and a mapping
$\Phi \in {\mathcal P}_Y$ such that $|u-v| < \delta$ but $d(\Phi (u),\Phi (v)) \ge \ve$. In other words,
the family \,${\mathcal P}_Y$ is not equicontinuous. According to the generalized
Arzel\'a theorem (see e.g.~\cite[pp.~119--120]{kolmogorovfomin}),
${\mathcal P}_Y$ cannot be relatively compact, so it is not compact.

\vskip .15cm

\noindent (c) Consider the point $y_0$ given in the hypothesis and suppose, by way of contradiction,
that \,${\mathcal P}^0_Y \ne \varnothing$. Then there are $f \in {\mathcal P}_Y$ and $r > 0$ such that
$B_\rho (f,r) \subset {\mathcal P}_Y$. In other words, if $g \in C(I,Y)$ and $\rho (g,f) < r$ then $g(I) = Y$.
On one hand, a neighborhood \,$V$ of $y_0$ can be found such that \,$B_d(y_0,r/2) \supset V$ and
\,$V \setminus \{y_0\}$ is arcwise connected. On the other hand, there is a closed ball \,$\ovl{B}_d(y_0,s) \subset V$.
Since $f$ is continuous, the set \,$f^{-1}(B_d(y_0,s))$ is open in $I$, so it is a countable union of pairwise disjoint intervals of
the form $(\al ,\beta )$, $[0,\beta )$ or $(\al ,1]$. In all three cases,
we have $f(\al ) \ne y_0 \ne f(\beta )$, and the continuity of $f$ implies
\,$f(\al ),f(\beta ) \in \ovl{B}_d(y_0,s)$. Then \,$f(\al ),f(\beta ) \in V \setminus \{y_0\}$, which is
arcwise connected. Therefore, in the first case, we can find a continuous mapping
\,$h = h_{\al ,\beta} :[\alpha ,\beta ] \to V \setminus \{y_0\}$ \,satisfying \,$h(\al ) = f(\al )$ \,and
\,$h(\beta ) = f(\beta )$.

\vskip .15cm

Define the mapping \,$g:I \to Y$ \,as follows: $g(t) = f(t)$ \,if \,$t \in I$ \break $\setminus \,f^{-1}(B_d(y_0,s))$,
$g(t) = h_{\al ,\beta} (t)$ if \,$t$ \,belongs to one of the intervals $(\al ,\beta )$ making up $f^{-1}(B_d(y_0,s))$,
$g(t) = f(\beta )$ if \,$t \in [0,\beta ) \subset f^{-1}(B_d(y_0,s))$, and \,$g(t) = f(\al )$ \,if \,$t \in (\alpha,1] \subset f^{-1}(B_d(y_0,s))$.
It is evident that \,$g$ \, is continuous and \,$g(t) \ne y_0$ \,for all $t \in I$. Then \,$g \not\in {\mathcal P}_Y$.
Now, the triangle inequality and the fact $s < r/2$ yield \,$d(g(t),f(t)) < r$ \,for all $t \in I$, so \,$g \in B_\rho (f,r)$.
This contradiction proves (c).

\vskip .15cm

\noindent (d) Fix $\al \in (0,1)$ and a ball $B_\rho (f,\ve ) \subset C(I,I^2)$. Define $f_0 := (1 - {\ve \over 2})f$.
Trivially, $f_0 \in C(I,I^2)$. Moreover,
$$
\rho (f,f_0) = \sup_{t \in I} \|f(t) - (1 - {\ve \over 2})f(t)\|_1 =
{\ve \over 2} \sup_{t \in I} \|f(t)\|_ 1 \le {\ve \over 2} < \ve ,
$$
so $f_0 \in B_\rho (f,\ve )$.
Take $g \in B_\rho (f_0,\al \rho (f,f_0))$. Then $d(g(t),f_0(t)) \le \al \rho (f,f_0) \le \al \ve /2$ for
all $t \in I$ and, by the triangle inequality,
$$
\|g(t)\|_1 \le \al \rho (f,f_0) + \|f_0(t)\|_1 \le
\al {\ve \over 2} +  1 - {\ve \over 2} < 1.
$$
Therefore $g(I) \ne I^2$, so
\,${\mathcal P} \cap B_\rho (f_0,\al \rho (f,f_0)) = \varnothing$. This had to be shown.
\end{proof}

\begin{remark}
{\rm Of course, the condition in (c) above is fulfilled if \,$Y = I^2$, but in this case the conclusion of (d) is stronger
than that of (c).
Notice that some assumption on \,$Y$ is really needed in order that \,${\mathcal P}_Y^0 = \varnothing$.
For instance, for the unit circle \,$Y = S^1 = \{(x,y) \in \R^2: \, x^2 + y^2\}$, which clearly does not satisfy
the mentioned condition, we have that  \,${\mathcal P}_{S^1}^0 \ne \varnothing$. Indeed, it is not difficult to show that
for the mapping \,$f:t \in I \mapsto (\cos (4 \pi t),\sin (4 \pi t)) \in S^1$ \,(which travels $S^1$ twice in the same
direction) one has \,$B_\rho (f,1/2) \subset {\mathcal P}_{S^1}$.}
\end{remark}

\begin{remark}
{\rm The last theorem yields that, topologically speaking, the Peano class is very small.
Another property of ${\mathcal P}_Y$ --easy to see and not related to the size-- is that it is {\it arcwise connected.}}
\end{remark}

The next statement tells us that, if we endow $C(I,I^2)$ with the semigroup structure given by coordinatewise multiplication,
then \,$\mathcal P$ has a chance to be considered large.

\begin{theorem} \label{Thm semigroupable}
The set \,${\mathcal P}$ is semigroupable.
\end{theorem}

\begin{proof}
Fix any sequence $(a_n)$ with $a_1 < a_2 < \cdots < a_n < \cdots \to 1$. Let $a_0 := 0$. According to Lemma \ref{Peano extrema prescribed}, we can find for every $n \in \N$ a mapping $f_n \in CS([a_{n-1},a_n],I^2)$ with $f_n(a_{n-1}) = (1,1) = f_n(a_n)$. Let us extend continuously $f_n$ to $I$ by defining $f_n(t) = (1,1)$ if $t \in I \setminus [a_{n-1},a_n]$. Then $f_n \in {\mathcal P}$ and, trivially, every power $f_n^m$ still belongs to ${\mathcal P}$. Consider the subsemigroup \,$G$ \,generated by $\{f_n\}_{n \ge 1}$. Given $\Phi \in G$, there exist $p \in \N$, $\{i_1 < \cdots < i_p\} \subset \N$ and $\{m_1, \dots ,m_p\} \subset \N$ satisfying \,$\Phi = f_{i_1}^{m_1} \cdots f_{i_p}^{m_p}$. Since
$$
I^2 \supset \Phi (I) \supset \Phi ([a_{i_p - 1},a_{i_p}]) = f_{i_p}^m ([a_{i_p - 1},a_{i_p}]) = I^2,
$$
we obtain \,$\Phi (I) = I^2$ or, that is the same, $G \subset {\mathcal P}$.
All that must be proved is that $G$ is infinitely generated. Assume, by way of contradiction, that there are finitely many elements of $G$ generating it. Taking into account the structure of $G$ and the fact that $G$ is commutative, there would be $p \in \N$ such that each $\Phi \in G$ can be written as $\Phi = f_1^{m_1} \cdots f_p^{m_p}$, for some $m_1, \dots ,m_p \in \{0,1,2, \dots\}$ depending on $\Phi$. But taking $\Phi = f_{p+1}$, the previous equality is not possible, because $f_j(t) = (1,1)$ for all $t \in [a_p,a_{p+1}]$ and all $j=1, \dots, p$. This is the desired contradiction.
\end{proof}

\begin{remark}
{\rm If \,$\mathcal P$ \,is considered as a subset of the additive {group\linebreak}
$(C(I,\R^2),+)$, then it is not very likely for
the sum of two given mappings in \,$\mathcal P$ to stay still in \,$\mathcal P$. Nevertheless, we can say at least the following:
given $N \in \N$, there are functions $f_1, \dots ,f_N \in {\mathcal P}$ such that $f_1 + \cdots + f_N \in {\mathcal P}$.
Indeed, for each $i \in \{1,...,N\}$ take as $f_i$ the mapping $\Phi$ provided in Lemma \ref{Peano extrema prescribed}
with $Y = I^2$, $[a,b] = [{i-1 \over N},{i \over N}]$, $u = (0,0) = v$, extended as $(0,0)$
to the remaining of $I$. But one cannot find a sequence
$\{f_n\}_{n=1}^\infty \subset {\mathcal P}$ such that $\sum_{n \ge 1} f_n$ converges uniformly to any function because, if this is were the case,
one would have $\lim_{n \to \infty} \sup_{t \in I} \|f_n(t)\|_1 = 0$, which is plainly not possible since
$\sup_{t \in I} \|f_n(t)\|_1 = 1$ for every $n$. We do not know whether there is a sequence
$\{f_n\}_{n=1}^\infty \subset {\mathcal P}$ such that $\sum_{n \ge 1} f_n$ converges {\it pointwise} to a function $f \in {\mathcal P}$.}
\end{remark}

\section{The family of space-filling curves}

\quad Throughout this section we shall deal with the algebraic size of the set $\mathcal{SF}$,
viewed as a subset of $C(I,\R^2)$. Recall that, under the distance given by \eqref{Eq1}, $C(I,\R^2)$ is an F-space, that is,
a completely metrizable topological vector space. In fact, it is a Banach space under the norm $\| \cdot \| := \rho (\cdot ,0)$.

\vskip .15cm

Concerning elementary topological properties, the set $\mathcal{SF}$ is clearly {\it non-closed} in $C(I,\R^2)$:
if we take $f \in {\mathcal P}$ then each $f_n := (1/n)f \in \mathcal{SF}$ $(n \ge 1)$ and ${f_n \to (0,0)} \not\in \mathcal{SF}$.
Moreover, $\mathcal{SF}^0 = \varnothing$. Indeed,
if $\varphi \in C(I,\R^2)$ and $\ve > 0$ are given, with $\varphi (t) = (g(t),h(t))$, then
from the uniform continuity of $g$ and $h$ one obtains an $N \in \N$ such that $|g(t)-g(u)| < \ve /2$ and $|h(t) - h(u)| < \ve /2$
whenever $t,u \in [{i-1 \over N},{i \over N}]$ $(i=1, \dots ,N)$. If we define $\widetilde{g}, \widetilde{h} : I \to \R$
as the polygonal functions joining successively the points $({i \over N},g({i \over N}))$ $(i=1,...,N)$ and, respectively, the points
$({i \over N},h({i \over N}))$ $(i=1,...,N)$, then the mapping $\widetilde{\varphi} (t) := (\widetilde{g}(t), \widetilde{h}(t))$
satisfies \,$\rho (\widetilde{\varphi}, \varphi ) < \ve$ \,and \,$\widetilde{\varphi} \not\in \mathcal{SF}$, so $\mathcal{SF}$
does not contain any $\rho$-ball.

\vskip .15cm

If $\varphi \in C(I,\R^2)$ then $\varphi (I)$ is compact, hence bounded and closed.
Then $\varphi (I) = \varphi (I)^0 \cup \partial \varphi (I)$. Therefore,
according to Definition \ref{def spacefilling curve} and the final paragraph of Section 2, we have that

\vskip .15cm

\centerline{\it $\varphi \in \mathcal{SF}$ \ if and only if \ $\la (\partial \varphi (I)) = 0$ and $(\varphi (I))^0 \ne \varnothing$.}

\vskip .15cm

We saw in Section 1 how Osgood's example provided a $\la$-space-filling curve \,$\psi$ \,that is {\it not} space-filling.
In this case, we have even that $(\varphi (I))^0 = \varnothing$; indeed, a continuous injective mapping $I \to \R^2$ cannot
fill in a square, see \cite{sagan}. In view of this, the following concept is in order.

\begin{definition}
{\rm A continuous mapping $\varphi : I \to \R^2$ is said to be a
{\it topologically space-filling curve} provided that \,$(\varphi (I))^0 \ne \varnothing$.
The family of all these mappings will be denoted by $\mathcal{TSF}$.}
\end{definition}

It is evident that $\mathcal{SF} \subset \mathcal{TSF} \subset \la$-$\mathcal{SF} := \{\la$-space-filling curves$\}$.
Moreover, both inclusions are strict. Indeed, for the second one we can appeal Osgood's example, while for the first one we
can construct on $[0,1/3]$ a curve filling $I^2$, and on $[2/3,1]$ an Osgood-type curve that is disjoint with $I^2$, and then to joint them
along $[1/3,2/3]$ by a segment so as to built a $\mathcal{TSF}$ mapping.

\vskip .15cm

In the following theorems it is shown that, in some algebraic senses,
our family \,$\mathcal{SF}$ \,can be thought as ``large''.

\begin{theorem} \label{Thm SF is spaceable}
The family \,$\mathcal{SF}$ is spaceable in $C(I,\R^2)$. In particular, it is maximal lineable.
\end{theorem}

\begin{proof}
Fix again any sequence $(a_n)$ with $a_1 < a_2 < \cdots < a_n < \cdots \to 1$.
By Lemma \ref{Peano extrema prescribed}, for every $n \in {\N}$ there is a mapping $f_n \in CS([a_{n-1},a_{n}],[-1,1]^2)$ with $f_n(a_{n-1}) = (0,0) = f_n(a_n)$, where $a_0 := 0$. Extend continuously each $f_n$ to $I$ by
setting $f_n (t) = (0,0)$ if $t \in I \setminus [a_{n-1},a_n]$. Since the supports of these functions are mutually disjoint,
Lemma \ref{Lemma-disjointsupports} tells us that $\{f_n\}_{n \ge 1}$ is a basic sequence of $C(I,\R^2)$. Define
$$
M := \ovl{\rm span} \{f_n: \, n \in \N\}.
$$
It is plain that $M$ is a closed vector subspace of $C(I,\R^2)$. Moreover, it is infinite dimensional because
the $f_n$'s, being members of a basic sequence, are linearly independent.

\vskip .1cm

Finally, let $f \in M \setminus \{0\}$.
Then there is a sequence $(c_n) \subset \R$ with some ${c_m} \ne 0$ such that $f = \sum_{n=1}^\infty c_nf_n$ in $C(I,\R^2)$.
Note that this series converges uniformly on $I$. Therefore $c_nf_n \to 0$ uniformly on $I$, that is,
$\lim_{n \to \infty} |c_n| \sup_{t \in I} \|f_n(t)\|_1 = 0$. But since $f_n(I) = [-1,1]^2$, we get $\sup_{t \in I} \|f_n(t)\|_1 = 1$
for all $n$, hence ${c_n} \to 0$. Therefore, there exists $p \in \N$ such that $|c_p| = \max \{|c_n|: \, n \in \N\} > 0$.
Consequently, $f(I) = \{(0,0)\} \, \cup \,\bigcup_{n \ge 1} (c_nf_n)([a_{n-1},a_{n}]) = \{(0,0)\} \, \cup \, \bigcup_{n \ge 1} |c_n|[-1,1]^2 =
|c_p|[-1,1]^2 =$ \break
$[-|c_p|,|c_p|]^2$. Then $f(I)$ is, trivially, Jordan measurable and satisfies $f(I)^0 \ne \varnothing$.
Thus, $f \in \mathcal{SF}$, as required. The maximal lineability of $\mathcal{SF}$ comes from the fact that dim$(M) = {\mathfrak c}$ ($= {\rm dim}\, (C(I,\R^2)$), because, by Baire's ca\-te\-go\-ry theorem, the dimension of any separable infinite dimensional F-space is ${\mathfrak c}$.
\end{proof}

\begin{proposition}
The family \,$\mathcal{SF}$ is dense in $C(I,\R^2)$.
\end{proposition}

\begin{proof}
Fix a ball $B_\rho (f, \ve )$. Since $f$ is uniformly continuous on $I$, there is $\delta > 0$ such that
$\|f(u) - f(v)\|_1 < \ve /2$ if $|u-v| < \delta$. Select a partition $\{0 = t_0 < t_1 < \cdots < t_N = 1\}$ with
$|t_j - t_{j-1}| < \delta$ $(j=1, \dots ,N)$. Then we have $\|f(t_j) - f(t_{j-1})\|_1 < \ve /2$ $(j=1, \dots ,N)$.
Choose any closed non-degenerate rectangle $R = [a,b] \times [c,d]$ with $\max \{b-a,d-c\} < \ve /2$ and
$\{f(t_0),f(t_1)\} \subset R$. Select also any mapping $\varphi \in CS([t_0,t_1],R)$ such that $\varphi (t_1) = f(t_1)$
(Lemma \ref{Peano extrema prescribed}). Define $g = (g_1,g_2)$ as $g|_{[t_0,t_1]} = \varphi$ and $g_1,g_2$ affine-linear in each
segment $[t_{j-1},t_j]$ $(j=2, \dots ,N)$, and such that $g(t_j) = f(t_j)$ for all $j$. It is easy to check that
$\varphi \in \mathcal{SF} \cap B_\rho (f, \ve )$, which shows the density of $\mathcal{SF}$.
\end{proof}

According to the last proposition and Theorem \ref{Thm SF is spaceable}, $\mathcal{SF}$ is dense and lineable.
However, this does not imply that $\mathcal{SF}$ is dense-lineable. In fact, we have been not able to prove this point,
yet our conjecture is the truthfulness of the claim. In view of this, we will content ourselves with showing the (maximal)
dense-lineability of the broader class $\mathcal{TSF}$. With this aim, the forthcoming two auxiliary assertions will reveal useful.

\begin{lemma} \label{Lemma-partially constant}
Let $(Y,d)$ be a locally arc-connected metric space, and let $t_0 \in I$. Then the set
$$
\mathcal{D}_{t_0} := \{\varphi \in C(I,Y): \, \varphi \hbox{ \rm \ is constant on some neighborhood } U=U_\varphi \hbox{\rm \ of } t_0 \}
$$
is dense in $C(I,Y)$, when this space is endowed with the uniform metric $\rho (f,g) = \sup_{t \in I} d(f(t),g(t))$.
\end{lemma}

\begin{proof}
Fix a ball $B_\rho (f,\ve ) \subset C(I,Y)$. Our goal is to show that $\mathcal{D}_{t_0} \cap B_\rho (f,\ve ) \ne \varnothing$.
Consider the ball $B_d(f(t_0),\ve /3) \subset Y$. By hypothesis, there is
a connected neighborhood $V$ of $f(t_0)$ in $Y$ such that $V \subset B_d(f(t_0),\ve /3 )$. Since $f$ is continuous at $t_0$,
there exists a neighborhood $[c,d]$ of $t_0$ in $I$ with $f([c,d]) \subset V$. We can suppose $0 < t_0 < 1$
(the case $t_0 \in \{0,1\}$ being easier to deal with), so that $c < t_0 < d$. Choose any $c',d'$ with $c < c' < t_0 < d' < d$. By local arc-connection,
we can find continuous mappings $g:[c,c'] \to V$, $h:[d',d] \to V$ satisfying $g(c) = f(c)$, $g(c') = f(t_0) = h(d')$ and
$h(d) = f(d)$. Let $U := [c',d']$ and define $\varphi : I \to Y$ as
$$
\varphi (t) =  \left\{
\begin{array}{ll}
                 f(t)  & \mbox{if } t \not\in [c,d] \\
                 g(t) & \mbox{if } t \in [c,c')  \\
                 f(t_0) & \mbox{if } t \in U \\
                 h(t)  & \mbox{if }  t \in (d',d].
\end{array} \right.
$$
Clearly $\varphi \in \mathcal{D}_{t_0}$. Moreover,
\begin{equation*}
\begin{split}
\rho (f,\varphi ) & = \sup_{t \in [c,d]} d(f(t),\varphi (t)) \\
                  & \le \sup_{t \in [c,d]} (d(f(t),f(t_0)) + d(f(t_0),\varphi (t))) \le \ve /3 + \ve /3 < \ve ,
\end{split}
\end{equation*}
due to the triangle inequality and the fact
$\varphi ([c,d]) = g([c,c']) \cup h([d',d]) \subset V \subset B_d(f(t_0),\ve /3 )$. Consequently,
$f \in \mathcal{D}_{t_0} \cap B_\rho (f,\ve )$, and we are done.
\end{proof}

\begin{lemma} \label{Lemma TSF1 spaceable}
The subfamily of \,$\mathcal{TSF}$ given by
$$
\mathcal{TSF}_1 := \{\varphi \in C(I,\R^2): \, (\varphi (U))^0 \ne \varnothing \, \hbox{\rm \ for all neighborhood } \,U \hbox{\rm \ of } \,1\}
$$
is spaceable in \,$C(I,\R^2)$.
\end{lemma}

\begin{proof}
We need a modification of the construction given in the proof of Theorem \ref{Thm SF is spaceable}.
Fix once more any sequence $(a_n)$ with $a_1 < a_2 < \cdots < a_n < \cdots \to 1$ and consider a partition of $\N$ into
infinitely many pairwise disjoint sequences $\{p(n,1) < p(n,2) < p(n,3) < \cdots \}$ $(n=1,2, \dots )$.
By Lemma \ref{Peano extrema prescribed}, for every pair $(n,k) \in \N \times \N$ there exists a mapping $g_{n,k} \in CS([a_{p(n,k)},a_{p(n,k)+1}],(1/k)I^2)$ with $g_{n,k}(a_{p(n,k)}) = (0,0) = g_{n,k}(a_{p(n,k){+1}})$.
Let us call $I_{n,k} := [a_{p(n,k)},a_{p(n,k)+1}]$ and extend continuously each $g_{n,k}$ on $I$ by defining it as $(0,0)$ on $I \setminus I_{n,k}$.
Now, fix $n \in \N$ and define $f_n := \sum_{k=1}^\infty g_{n,k}$. Note that this series is in fact a finite sum at each point $t \in I$, so it is well defined. If $t < 1$ there is a neighborhood of $t$ lying at most on two intervals $I_{n,k}$ $(k=1,2, \dots )$, which entails the continuity of $f_n$ at $t$. Observe that the continuity at $t=1$ is guaranteed by the fact $g_{n,k}(I_{n,k}) = k^{-1} I^2$ for all $k$, from which we conclude that each $f_n$ is continuous on $I$. Since the supports of the functions $f_n$ $(n \ge 1)$ are mutually disjoint, Lemma \ref{Lemma-disjointsupports}
tells us that they form a basic sequence.

\vskip .15cm

As in the proof of Theorem \ref{Thm SF is spaceable}, define
$$
M := \ovl{\rm span} \, \{f_n\}_{n \ge 1}.
$$
Then $M$ is a closed infinite dimensional vector subspace of $C(I,\R^2)$. Let $f \in M \setminus \{0\}$. Then there are uniquely determined
real coefficients $c_1,c_2, \dots$ with some $c_m \ne 0$ such that $f = \sum_{n=1}^\infty c_n f_n$, where the convergence of the series is
uniform on $I$. Fix a neighborhood $U$ of $t=1$ in $I$. Since $a_{p(m,k)} \to 1$ as $k \to \infty$, we can find $k_0 \in \N$ such that
$[a_{p(m,k_0)},a_{p(m,k_0)+1}] \subset U$. Therefore
\begin{equation*}
\begin{split}
f(U) &\supset f([a_{p(m,k_0)},a_{p(m,k_0)+1}]) = c_m f_m([a_{p(m,k_0)},a_{p(m,k_0)+1}]) \\
     &= c_m g_{p(m,k_0)}([a_{p(m,k_0)},a_{p(m,k_0)+1}]) = c_m k_0^{-1} I^2,
\end{split}
\end{equation*}
hence $(f(U))^0 \ne \varnothing$. In other words, $f \in \mathcal{TSF}_1$, which shows the desired spaceability.
\end{proof}

Of course, the last construction can be reproduced for any fixed $t_0 \in I$, but $t_0 = 1$ is enough for us.

\begin{theorem}
The family \,$\mathcal{TSF}$ is maximal dense-lineable in $C(I,\R^2)$. Hence the family $\la$-$\mathcal{SF}$ is
maximal dense-lineable as well.
\end{theorem}

\begin{proof}
It is enough to prove that the subfamily $\mathcal{TSF}_1$ defined in Lemma \ref{Lemma TSF1 spaceable}
is maximal dense-lineable in $C(I,\R^2)$.
With this aim, observe first that $A := \mathcal{TSF}_1$ is $\mathfrak{c}$-lineable in $C(I,\R^2)$ by the mentioned lemma and Baire's theorem.
Secondly, the set
$$
B := \{\varphi \in C(I,\R^2): \, \varphi \hbox{ \,is constant on some neighborhood } U=U_\varphi \hbox{ of }1 \}
$$
is dense in $C(I,\R^2)$, due to Lemma \ref{Lemma-partially constant}. Trivially, $B$ is also a vector space, whence $B$ is dense-lineable. It is also straightforward that
$A \cap B = \varnothing$. Finally, if $\varphi \in A$, $\psi \in B$ and $U$ is a neighborhood of $1$, there are a neighborhood $V \subset U$ of $1$ and a constant $C \in \R^2$
such that $\psi (t) = C$ for all $t \in V$ and $(\varphi (V))^0 \ne \varnothing$. Then
$$
((\varphi + \psi ) (U))^0 \supset ((\varphi + \psi ) (V))^0 = C + (\varphi (V))^0 \ne \varnothing .
$$Thus $((\varphi + \psi ) (U))^0 \ne \varnothing$, that is, $\varphi + \psi \in A$ and $A+B \subset A$.
The proof is finished after a direct application of Theorem \ref{Thm-denselineability-OrdonezLBG} with $E := C(I,\R^2)$ and $\al := \mathfrak{c}$.
\end{proof}

We conclude this paper with the following theorem. Recall that the vector space $C(I,\R^2)$
becomes an algebra if the multiplication is defined coordinatewise.

\begin{theorem} \label{Thm SF is strongly algebrable}
The family \,$\mathcal{SF}$ is strongly algebrable.
\end{theorem}

\begin{proof}
As a first step, we construct an appropriate sequence $\{f_n\}_{n \ge 1}$ genera\-ting a free algebra in $C(I,\R^2)$.
By Lemma \ref{Peano extrema prescribed}, there exists a Peano curve $\varphi \in CS(I,{[-1,1]}^2)$ such that $\varphi (0) = (0,0) = \varphi (1)$.
If $T = [a,b] \subset \R$ is an interval, we define $\varphi_T : [a,b] \to \R^2$ as $\varphi_T(t) = \varphi \big({t-a \over b-a} \big)$, so
that $\varphi (T) = {[-1,1]}^2$ and $\varphi_T(a) = (0,0) = \varphi_T(b)$. Denote by $\Q$ the set of rational numbers and
consider the countable set $J = \{\sigma_k\}_{k \ge 1}$ defined as
$$
J := \{\sigma = (q_1,q_2,..,q_j,0,0,...) \in (0,+\infty )^\N : \, q_1,..., q_j \in \Q \cap (-1,1) , \, j \in \N\}.
$$
Take a sequence $(a_n)$ with $0 < a_1 < \cdots < a_n < \cdots \to 1$
and consider the sequence of intervals $[a_n,a_{n+1}]$ $(n \ge 1)$. Then we can extract from it infinitely many countable families of
sequences of intervals $\{I_{n,k} : \, k \in \N\}$ $(n \in \N )$ such that $I_{n,k} \cap I_{m,l} = \varnothing$ as soon as $(n,k) \ne (m,l)$ and,
for every $n \in \N$, the intervals $I_{n,k}$ approach $1$ as $k \to \infty$. Split each interval $I_{n,k}$ into three segments of equal length, say
$I_{n,k} = I_{n,k,1} \cup I_{n,k,2} \cup I_{n,k,3}$, where $I_{n,k,2}$ is the middle segment.

\vskip .15cm

Fix $n \in \N$ and define the mapping
$f_n : I \to \R^2$ as follows. For all $k \in \N$, we set $f_n := k^{-1} \varphi_{I_{n,k,2}}$ on $I_{n,k,2}$
and $f_n := (0,0)$ on $I_{n,k,1} \cup I_{n,k,3}$. If $m \ne n$ then we set $f_n := (k^{-1} q_n, k^{-1} q_n)$ on $I_{m,k,2}$,
where $q_n$ is the $n$th component of the sequence $\sigma_k \in J$. Both components of $f_n$ are defined as
affine linear on $I_{m,k,1}$ and $I_{m,k,3}$, with value $(0,0)$ at the left endpoint of $I_{m,k,1}$ and at the right endpoint of $I_{m,k,3}$.
Finally, set $f_n := (0,0)$ on $I \setminus \bigcup_{k=1}^\infty I_{n,k}$.
Each $f_n$ is clearly continuous on $[0,1)$, while its continuity at $t=1$ (where $f_n$ takes the value $(0,0)$) is guaranteed
by the fact that $\sup_{t \in I_{n,k}} \|f_n(t)\|_1 \le k^{-1}$ for all $k \in \N$.

\vskip .15cm

Now, let $N \in \N$ and consider a nonzero polynomial $P$ of $N$ variables without constant term, say $P(x_1, \dots ,x_N)$.
Without loss of generality, we may assume that $x_N$ appears explicitly in $P$, so that there is $m \in \N$ and polynomials
$Q_j$ $(j=0,1, \dots ,N-1)$ of $N-1$ real variables, with $Q_m \not \equiv 0$, such that
$$
P(x_1, \dots ,x_N) = \sum_{j=0}^m Q_j(x_1, \dots ,x_{N-1}) \, x_N^j .
$$
Let $F := P(f_1, \dots ,f_N)$.
Our aim is to show that {$F \in \mathcal{SF}$ (it must also be proved that $F\not\equiv0$, but this is unnecessary because $0\notin\mathcal{SF}$)}.

\vskip .15cm

Assume first that $Q_m(0, \dots ,0) \ne 0$. Since $Q_m$ is continuous, there is $r \in (0,1)$ such that $Q_m{(}x_1, \dots ,x_{N-1}) \ne 0$
for all $(x_1, \dots ,x_{N-1}) \in (-r,r)^{N-1} \setminus \{(0, \dots ,0)\}$. Taking $p \in \N$ with $1/p < r$ and $q_j := 1/p$
$(j=1, \dots ,N-1)$, we get the existence of a point $(q_1, \dots ,q_{N-1}) \in (\Q \cap (-1,1))^{N-1}$ such that
\begin{equation}\label{Eq2}
Q_m(k^{-1}q_1, \dots ,k^{-1}q_{N-1}) \ne 0 \hbox{ \,\,for all } \,k \in \N .
\end{equation}
If, on the contrary, we had $Q_m(0, \dots ,0) = 0$, then
we would get a point $q = (q_1, \dots ,q_{N-1}) \in (\Q \cap (-1,1))^{N-1}$ satisfying \eqref{Eq2} too.
In order to see this, assume, by way of contradiction, that for each point $p=(p_1, \dots ,p_{N-1})$
$\in (\Q \cap (-1,1))^{N-1}$
there are
infinitely many $t \in \R$ with $Q_m(t p_1, \dots , t p_{N-1}) = 0$. Since the left hand side of the latter equation
is a polynomial in the variable $t$, we would have $Q_m(t p_1, \dots , t p_{N-1}) = 0$ for all $t$. Fixing $t$
and taking into account the density of \,$\Q \cap (-1,1)$ \,in $(-1,1)$ and the continuity of $Q_m$, we get
$Q_m(t x_1, \dots , t x_{N-1}) = 0$ for all $(x_1, \dots ,x_N) \in (-1,1)^{N-1}$ and all $t \in \R$, so $Q_m \equiv 0$,
a contradiction. Hence there is $p \in (\Q \cap (-1,1))^{N-1}$ such that the set of \,$t \in \R$ \,for
which \,$Q_m(t p_1, \dots , t p_{N-1}) = 0$ \,is finite. Since $0$ is one of such $t$'s, there is $s \in \N$ with
$Q_m(t p_1, \dots , t p_{N-1}){\ne} 0$ for all $t \in (0,1/s]$. Therefore we get \eqref{Eq2} if we set
$q = (s^{-1}p_1, \dots , s^{-1}p_{N-1})$.

\vskip .15cm

Let $\sigma := (q_1, \dots ,q_{N-1},0,0, \dots ) \in J$, where $(q_1, \dots ,q_{N-1}) \in$ \break
$(\Q \cap (-1,1))^{N-1}$ satisfies \eqref{Eq2}.
Then there is $k \in \N$ such that $\sigma = \sigma_k$. Consider the interval $I_{N,k}$ and its subinterval $I_{N,k,2}$.
It happens that, for every $t \in I_{N,k,2}$,
%

\begin{align}\label{Eq3}
\!\!F(t) &= P(f_1(t), \cdots ,f_N(t)) \nonumber\\
     &= P((k^{-1}q_1,k^{-1}q_1), \dots ,(k^{-1}q_{N-1},k^{-1}q_{N-1}),k^{-1} \varphi_{I_{N,k,2}}(t)) \nonumber\\
     &= \sum_{j=0}^m Q_j((k^{-1}q_1,k^{-1}q_1), \dots ,(k^{-1}q_{N-1},k^{-1}q_{N-1})) (k^{-1} \varphi_{I_{N,k,2}}(t))^j.
\end{align}

{Recall that, given any polynomial $H(x_1, \dots ,x_N)$, we have that}\linebreak $H((a_1,b_1), \cdots ,(a_N,b_N)) = (H(a_1, \dots ,a_N),H(b_1, \dots ,b_N))$.
By the definition of the $f_n$'s, the image $L_2 := F(I \setminus I_{N,k})$ is the union of two piecewise continuously differentiable curves
in $\R^2$, so having empty interior $L_2^0$ (hence $L_2 = \partial L_2$) and Lebesgue measure $\la (\partial L_2) = 0$. Thanks to \eqref{Eq3},
the set $L_1 := F(I_{N,k})$ is the image of the square $R := [-1/k,1/k]^2$ under the $C^1$-mapping $S:\R^2 \to \R^2$ given by
$S(x,y) = (H(x),H(y))$, where $H$ is the nonconstant polynomial $H(x) = \sum_{j=0}^m \al_j \, x^j$, with
$\al_j := Q_j(k^{-1}q_1, \dots ,k^{-1}q_{N-1})$ (it is nonconstant because, from \eqref{Eq2}, $\al_m \ne 0$; this also yields $F \not\equiv 0$).
Therefore there is a point $x_0 \in (-1/k,1/k)$ such that $H'(x_0) \ne 0$, so the determinant of the Jacobian
matrix $J_S (x,y)$ of the transformation $S$ \,at $(x_0,x_0) \in R$ is $H'(x_0)^2 \ne 0$. By the inverse mapping theorem,
$S$ \,has a local differentiable (hence continuous) inverse at $(x_0,x_0)$, and so $S$ \,is locally open at this point, which
yields $L_1^0 = (S(R))^0 \ne \varnothing$.

\vskip .15cm

Finally, since $L_1$ is compact (hence closed) in $\R^2$, one has that
$\partial L_1 \subset L_1 = S(R) = S(R^0) \cup S(\partial R)$. Since $S$ is locally open at those points $(x,y) \in R^0$ with
det$\,J_S(x,y) \ne 0$, we deduce that $\partial L_1 \subset S(C) \cup S(\partial R)$, where $C := \{(x,y) \in R^0: \, {\rm det} \,J_{S}(x,y) = 0\}$.
On one hand, since $S$ is continuously differentiable on $\R^2$, Sard's theorem (see e.g.~\cite[p.~47]{sternberg})
tells us that $\la (S(C)) = 0$. On the other hand, the continuous differentiability of $S$ on $\R^2$
implies the well-known estimation
$$
\la (S(\partial R)) \le \int_{\partial R} |{\rm det} \,J_S| \, d\la \le \sup_{\partial R}|{\rm det} \,J_S| \cdot \la (\partial R) = 0.
$$
Thus, $\la (S(\partial R)) = 0$, hence $\la (\partial L_1) = 0$. To sum up, we get $(F(I))^0 \supset L_1^0 \ne \varnothing$
and
$$
\la (\partial F(I)) \le \la ((\partial L_1) \cup (\partial L_2))
\le \la (\partial L_1) + \la (\partial L_2) = 0.
$$
This entails $(F(I))^0 \ne \varnothing$ \,and \,$\la (\partial F(I)) = 0$.
In other words, $F \in \mathcal{SF}$, which finishes the proof.
\end{proof}

\begin{remark}
{\rm The mere {\it algebrability} of $\mathcal{SF}$ can be obtained in an easier way as follows.
Consider a sequence $\{a_1 < a_2 < \cdots \} \subset [0,1)$ and the intervals $I_n = [a_n,a_{n+1}]$ $(n \ge 1)$.
By Lemma \ref{Peano extrema prescribed}, for every $n \in \N$ there is $g_n \in CS(I_n,I^2)$ such that $g_n(a_n) = (0,0) = g_n(a_{n+1})$.
Define the continuous function $f_n:I \to \R^2$ as $g_j$ on $I_j$ $(j=1, \dots ,n)$ and $(0,0)$ on $I \setminus \bigcup_{j=1}^n I_j$.
Let ${\mathcal A}$ denote the algebra ge\-ne\-ra\-ted by the $f_n$'s. Then ${\mathcal A}$ is infinitely ge\-ne\-ra\-ted,
because each $f_n$ cannot be written as
$P(f_1, \dots ,f_{n-1})$, $P$ being a nonconstant polynomial in $n-1$ real variables: indeed, such a function $P(f_1, \dots ,f_{n-1})$
would be zero on $I_n$, which is absurd since $f_n = g_n$ on $I_n$. Now, fix $N$ and a nonzero polynomial $P(x_1, \dots ,x_N)$ of $N$ real variables.
It must be proved that the mapping $F := P(f_1, \cdots ,f_N)$ either is identically $(0,0)$ or belongs to $\mathcal{SF}$ (observe that
$F \not\equiv (0,0)$ is not demanded; in fact, ${\mathcal A}$ is {\it not} a {\it free} algebra because, for instance, the nonzero polynomial
without constant term $P(x,y) := x^2y - xy^2$ satisfies $P(f_1,f_2) \equiv (0,0)$). Without loss of generality, it can be assumed that $f_N$
appears explicitly in the expression of $P(x_1, \dots ,x_N)$ as sum of monomials $x_1^{m_1} \cdots x_N^{m_N}$. Consider the
one-variable polynomials without constant term $P_1(x) := P(x,x, \dots ,x)$, $P_2(x) := P(0,x, \dots ,x)$, $P_3(x) := P(0,0,x, \dots ,x),
\dots , P_N(x) :=$ \break
$P(0,0, \dots ,0,x)$. According to the definition of the $f_n$'s, we have that $F$ equals $P_j (g_j)$ on $I_j$.
Therefore $F(I) = \bigcup_{j=1}^N S_j(I^2)$, where $S_j(x,y) := (P_j(x),P_j(y))$ $(j=1, \dots ,N)$. If $P_j$ is constant then $P_j \equiv 0$, so
$S_j(I^2) = \{(0,0)\}$. If $P_j$ is not constant then the same Sard-change-of-variable argument of the final part of the proof of Theorem
\ref{Thm SF is strongly algebrable} leads us to $\la(\partial S_j(I^2)) = 0$ and $(S_j(I^2))^0 \ne \varnothing$.
Hence either $F \equiv (0,0)$ or $\la(\partial F(I)) = 0$ and $(F(I))^0 \ne \varnothing$, as required.
}
\end{remark}



\end{document}